\documentclass[pdflatex,sn-mathphys-num]{sn-jnl} 
\usepackage{graphicx}%
\usepackage{multirow}%
\usepackage{amsmath,amssymb,amsfonts}%
\usepackage{amsthm}%
\usepackage{mathrsfs}%
\usepackage[title]{appendix}%
\usepackage{xcolor}%
\usepackage{textcomp}%
\usepackage{manyfoot}%
\usepackage{booktabs}%
\usepackage{algorithm}%
\usepackage{algorithmicx}%
\usepackage{algpseudocode}%
\usepackage{listings}%
\theoremstyle{thmstyleone}%
\newtheorem{theorem}{Theorem}%
\newtheorem{proposition}[theorem]{Proposition}%

\theoremstyle{thmstyletwo}%
\newtheorem{example}{Example}%
\newtheorem{remark}{Remark}%
\newtheorem{conjecture}{Conjecture}%

\theoremstyle{thmstylethree}%

\usepackage{caption}
\usepackage{soul}
\usepackage{subcaption}
\usepackage{seqsplit}
\usepackage[skins,breakable]{tcolorbox}

\definecolor{codebgcolour}{rgb}{0.94,0.94,0.96}
\definecolor{codegreen}{rgb}{0,0.5,0}
\definecolor{codemagenta}{rgb}{0.9,0,0.5}

\newcommand{\is}[1]{{\ttfamily\seqsplit{#1}}} 

\newtcolorbox[auto counter,number within=section]{procbox}[1]{
	enhanced,
	breakable,
	colframe=blue!75!green!75,
	fonttitle=\bfseries\boldmath,
	title={#1}
}

\newcommand{\twprod}{\mathbin{
		\ooalign{\raise1.15ex\hbox{$\scriptstyle\sim$}\cr\hidewidth$\times$\hidewidth\cr}
}}

\newcommand{\tri}{\mathcal{T}}
\newcommand{\bz}{\mathbb{Z}}
\newcommand{\lk}{\mathrm{lk}}
\newcommand{\cpp}{\mathbb{C}P^2}
\newcommand{\Chains}{C}
\newcommand{\Boundaries}{B}
\newcommand{\Cycles}{Z}

\theoremstyle{remark}
\newtheorem{question}{Question}
\newtheorem{housekeeping}{Housekeeping}

\usepackage{cleveref}

\begin{document}

\title[The Census of 4-Manifold Triangulations]{Small Triangulations of 4-Manifolds and the 4-Manifold Census}

\author[1]{\fnm{Rhuaidi} \sur{Burke}}\email{rhuaidi.burke@uq.edu.au}

\author[1]{\fnm{Benjamin} \sur{Burton}}\email{bab@maths.uq.edu.au}

\author[2]{\fnm{Jonathan} \sur{Spreer}}\email{jonathan.spreer@sydney.edu.au}

\affil[1]{\orgdiv{School of Mathematics and Physics}, \orgname{The University of Queensland}, \orgaddress{\city{Brisbane}, \postcode{4072}, \state{Queensland}, \country{Australia}}}

\affil[2]{\orgdiv{School of Mathematics and Statistics}, \orgname{The University of Sydney}, \orgaddress{\city{Sydney}, \postcode{2006}, \state{New South Wales}, \country{Australia}}}

\abstract{We present a framework to classify PL-types of large censuses of triangulated $4$-manifolds, which we use to classify the PL-types of all triangulated $4$-manifolds with up to six pentachora. This is successful except for triangulations homeomorphic to the $4$-sphere, $\mathbb{C}P^2$, and the rational homology sphere $QS^4(2)$, where we find at most four, three, and two PL-types respectively. We conjecture that they are all standard. In addition, we look at the cases resisting classification and discuss the combinatorial structure of these triangulations---which we deem interesting in their own rights.}



\maketitle

\section{Introduction}\label{sec:intro}	
In the context of computational topology it is desirable to have a data set of examples on which to perform experiments and test hypotheses. As such, an exhaustive list of all $4$-manifold triangulations of a certain size and type, called a {\em census}, serves as such a useful reference. In dimension three we currently have censuses available with up to $11$ tetrahedra (containing $13\,400$ closed, orientable, prime, minimal triangulations). In dimension $4$, where a closed $4$-manifold requires an even number of pentachora, censuses are currently limited to triangulations with $2$, $4$, and $6$ pentachora. 
However, despite this relative scarcity of data, even with just three tiers of the census, there are a total of $441\,287$ triangulations to consider.

Given a smooth manifold $X$, another smooth manifold $X'$ is called {\em exotic (with respect to $X$)}, if $X'$ is homeomorphic but not diffeomorphic to $X$. In other words, $X$ and $X'$ represent the same topological manifold but are distinct as smooth manifolds. Dimension four is the first dimension in which exotic smooth structures appear. It is a fundamental problem in $4$-manifold topology to determine the number of smooth structures a particular $4$-manifold admits. The simplest and most famous instance of this problem is the \emph{Smooth $4$-Dimensional Poincar\'e Conjecture} (S4PC for short):

    \begin{conjecture}[S4PC]
        There do not exist exotic $4$-spheres.
    \end{conjecture}

It is typical to discuss exotic structures in relation to a `standard' or canonical smooth structure. For example, the standard structure on $\mathbb{R}^n$ is the one given by a single chart with the identity map; on $\mathbb{S}^n$ it is the one with two charts given by stereographic projection.

\begin{remark}
The concept of a standard smooth structure is not well-defined in isolation. For instance, let $K3$ be the {\em $K3$ surface} with the standard smooth structure coming from being the zero-set of $w^4+x^4+y^4+x^4=0$ in $\mathbb{C}P^3$, and let $\mathbb{C}P^2$ be the standard complex projective plane. Then each of the connected sums\footnote{Given two $d$-manifolds $X$ and $Y$, their {\em connected sum}, written $X\# Y$ is formed by removing an open $d$-ball from each of $X$ and $Y$ and gluing them together along their resulting boundaries. We use $\#_k X$ to denote the $k$-fold connected sum of $X$ with copies of itself. The operation is independent of the $d$-ball removed from their summands.} $X_1 = \#_{3} \cpp \#_{20} \overline{\cpp}$ and $X_2 = K3 \# \overline{\cpp}$ inherits a standard smooth structure from their summands. It is known that $X_1\cong_{\mathrm{TOP}}X_2$ but $X_1\not\cong_{\mathrm{DIFF}}X_2$ \cite{kronheimerMrowka-K3CP2}, meaning $X_1$ is exotic with respect to $X_2$ and vice versa. The smallest triangulations of two $4$-manifolds that are known to be homeomorphic, but not diffeomorphic, are {\em ideal triangulations} (i.e.\ triangulations of manifolds with non-empty boundary where the boundary is given by a neighbourhood of a vertex) with $10$ pentachora \cite{Burke-SoftwareJoCGver}. The authors are not aware of any triangulations of closed $4$-manifolds which are homeomorphic, but not diffeomorphic, and which are not simply formed from a connected sum of ‘standard’ well-known manifolds (such as $X_1$ and $X_2$ above). 
\end{remark}

Cairns \cite{Cairns-TriangulationsC1-35,Cairns-Triangulations61} and Whitehead \cite{Whitehead-C1complexes} show that every smooth $n$-manifold can be triangulated, that is, admits a piecewise-linear (PL) structure. Moreover, every PL $n$-manifold for $n\leq 6$ admits a compatible smooth structure which is unique up to diffeomorphism~\cite{HirschMazur,Munkres-Smoothing}. Hence, there is a bijective correspondence between isotopy classes of smooth and PL structures on 4-manifolds, and so we speak of smooth and PL structures on a 4-manifold interchangeably. 

\paragraph*{Related Work.} The past decade has seen significant interest in the problem of enumerating and classifying triangulations, including several other attempts at classifying and simplifying triangulations of $4$-manifolds, particularly $4$-spheres. Most recently, P\'erez-Cerezo \cite{perez}, building on the work of Joswig et al.\  \cite{joswigLofanoLutzTsuruga-sphereRecog}, analysed the same census we investigate in this paper. 

More broadly, there has been a growing body of utilising Pachner moves to either simplify triangulations or establish PL-homeomorphisms. For example, the work Bj\"orner and Lutz \cite{lutzBistellarFlips}, Lutz and Tsuruga \cite{tsurugaLutz2013constructingcomplicatedspheres}, the second author \cite{burton2011pachner}, Altmann and the third author \cite{AltmannMCMC}, the second and third author \cite{burtonSpreerK3yrf}, and more recently the first author \cite{Burke-SoftwareJoCGver}. We note that whilst Bj\"orner, Lutz, Joswig et al., and Tsuruga, and (to some extent) P\'erez-Cerezo use \emph{simplicial complexes}, the work of the authors and Altmann instead use \emph{generalised triangulations} (see \Cref{sec:prelims-triangulations}). These are more `flexible' than simplicial complexes since, for example we allow two facets of the same simplex to be identified.

\paragraph*{Contributions.} We study the census of triangulations of closed, orientable, $4$-manifolds with up to $6$ pentachora. 
This data set contains 8 triangulations with 2 pentachora, 784 triangulations with 4 pentachora, and 440 495 triangulations with 6 pentachora.
It includes a variety of triangulations with interesting combinatorial and topological features (see \Cref{tab:2p-TOPclassTable,tab:4p-TOPclassTable,tab:6p-TOPclassTable} for a breakdown of numbers of triangulations by (PL-)topological types). 
The dataset was first generated by Budney and the second author \cite{census} using \texttt{tricensus}, a utility of \cite{Regina}. 

In addition, we present a new algorithm for finding PL-homeomorphisms between large numbers of $4$-manifold triangulations by combining original heuristics and computational tools. This algorithm yields a near complete classification of the PL-types in the census. 
We note that, due to the undecidability of the $4$-manifold homeomorphism problem \cite{Markov58}, we can only hope for heuristics which for as many cases as possible, give the correct answer, in as short a time as possible.
We also investigate the combinatorial structure of the triangulations resisting classification.

Specifically, we improve on the results of P\'erez-Cerezo, reducing the upper bounds on the number of potential PL classes for $\cpp$ and $\mathbb{S}^4$: within the set of 6-pentachoron triangulations homeomorphic to $\cpp$, we reduce the number of potential PL classes from 5 down to 3; within the set of 4-pentachoron triangulations homeomorphic to $\mathbb{S}^4$, we reduce from 3 to 2; and finally for the set of 6-pentachoron triangulations homeomorphic to $\mathbb{S}^4$, we obtain at most 4 potential classes, down from 36 in \cite{perez}.

\paragraph*{Acknowledgements.} The authors would like to thank Ryan Budney for his contributions at an early stage of this project. The authors would also like to thank the anonymous referees for insightful comments that improved the presentation of this paper. This paper was finished whilst the first author was visiting the University of Sydney, and thanks the University of Sydney for their hospitality. The third author is supported by the Australian Research Council’s Discovery funding scheme (project no. DP190102259).

\section{Preliminaries}\label{sec:bg}

In \Cref{sec:prelims-triangulations} and \Cref{sec:prelims-localMoves} we review some of the basic theory of (generalised) triangulations and local moves on triangulations. In \Cref{sec:prelims-4mflds} and \Cref{sec:prelims-handles} we recap some of the classical results in $4$-manifold theory and handle decompositions.  
        		
\subsection{Triangulations}\label{sec:prelims-triangulations}

We refer to a 4-simplex as a \emph{pentachoron} (plural:\ \emph{pentachora}). The tetrahedral cells of a pentachoron are referred to as \emph{facets}. We label the vertices of a pentachoron by elements of $\{0,1,2,3,4\}$, and use the convention that facet $i$ refers to the facet opposite to the vertex labelled by $i$. A \emph{(generalised, $4$-dimensional) triangulation} $\tri$ is a finite collection of $n$ abstract pentachora, some or all of whose $5n$ facets are affinely identified (`glued') in pairs. More precisely, let $\widetilde{\Delta}=\{\Delta_0,\Delta_1,\ldots,\Delta_{n-1}\}$ be a set of $n$ pentachora, and let $\Phi=\{\varphi_0,\ldots,\varphi_{m-1}\}$ be a set of at most $m\leq 2n$ \emph{face gluings}, such that each $\varphi_i$ is an affine identification between two distinct facets of simplices, and each facet is a part of at most one such identification. We write $\Delta_i(a)$ to denote vertex $a$ of pentachoron $\Delta_i$, and $\Delta_i(abcd)$ to denote the facet ${abcd}$ of $\Delta_i$ (and analogous notation for edges and triangles). A face gluing $\varphi\in\Phi$ is then explicitly described by an expression of the form $\Delta_i({abcd})\leftrightarrow \Delta_j({efgh})$, which means that facet ${abcd}$ of $\Delta_i$ is mapped to facet ${efgh}$ of $\Delta_j$ such that $a\leftrightarrow e$, $b\leftrightarrow f$, $c\leftrightarrow g$, and $d\leftrightarrow h$. Our triangulation is defined to be $\mathcal{T}:=\widetilde{\Delta}/\Phi$ the identification space obtained under the natural quotient map $q:\widetilde{\Delta}\to \mathcal{T}$. In other words, the data specifying a triangulation are the triple $(\widetilde{\Delta},\Phi,q)$. The \emph{(real) boundary} of $\tri$ consists of all the facets that are not identified with any other facets. Figure \ref{fig:triangulations} depicts a typical triangulation.

\begin{remark}
	Such triangulations are also referred to as an \emph{unordered $\Delta$-complex}.
\end{remark}

\begin{remark}
These triangulations are typically far more efficient than simplicial complexes since, for example, we allow two facets of the same pentachoron to become identified.
\end{remark}

\begin{remark}
	Our triangulations as defined above are not \emph{a priori} triangulations ``of something'' (for example a manifold); they are, initially at least, purely abstract combinatorial objects. One must check certain conditions hold to conclude that the triangulation has the structure of a manifold (for example, that the link---defined below---of each vertex is homeomorphic to either $\mathbb{S}^3$ or $B^3$).
\end{remark}

The gluings defining $\tri$ also have the effect of merging vertices, edges, triangles, and tetrahedra of the pentachora into equivalence classes, which we refer to as the \emph{vertices}, \emph{edges}, \emph{triangles}, and \emph{tetrahedra} of $\tri$.	

The \emph{link} $\lk(v)$ of a vertex $v$ of $\tri$, is the `frontier' of a small regular neighbourhood of $v$. We treat vertex links as triangulated $3$-dimensional spaces, formed by inserting a small tetrahedron into each corner of each pentachoron, and then joining together the tetrahedra from adjacent pentachora along their triangular faces. This mirrors the traditional concept of a link in a simplicial complex, but is modified to support generalised triangulations.

If $v$ lies in the boundary of $\tri$, its link is a space with boundary. If the link is homeomorphic to $B^3$, then we refer to $v$ as a \emph{boundary vertex}. If the link is any other space with boundary then we say $v$ is an \emph{invalid vertex}. On the other hand, if $v$ does not lie in the boundary of $\tri$, its link is a closed space. If the link is (PL-)homeomorphic to $\mathbb{S}^3$, then $v$ is an \emph{internal vertex}, and if it is any other closed space then it is referred to as an \emph{ideal vertex}.

We insist that no edge of $\tri$ is identified with itself in reverse, and that no triangle is identified with itself via a non-identity permutation. This, together with the condition that the link of every vertex is either $B^3$ or $\mathbb{S}^3$, guarantees that $\tri$ triangulates a $4$-manifold.

Given a triangulation $\tri$ of a $4$-manifold, the vector $f(\tri) = (f_0, f_1, f_2, f_3, f_4)$, where $f_i$ denotes the number of $i$-dimensional faces in $\tri$, is called its {\em face vector}, or {\em $f$-vector} for short. Since $\tri$ triangulates a $4$-manifold, it satisfies the following equations
\begin{eqnarray}
    -2f_3+5f_4 &=& 0 \label{eq:ds1}\\
    2f_1 - 3f_2 + 4f_3 - 5f_4 &=& 0 \label{eq:ds2}\\
    f_0 - f_1 + f_2 - f_3 + f_4 &=& \chi(\tri) \label{eq:ds3}
\end{eqnarray}
known as {\em generalised Dehn--Sommerville equations} \cite{dehnSommervilleEqns}. Here, $\chi(\tri)$ denotes the {\em Euler characteristic} of $\tri$, a topological invariant. Observe that Equation \eqref{eq:ds1} implies that a triangulation of a closed $4$-manifold must have an even number of pentachora.

To encode a triangulation, we give each pentachoron a label and an ordering of its five vertices. Two triangulations are \emph{(combinatorially) isomorphic} if they are identical up to relabelling of pentachora and/or reordering of the pentachoron vertices. We can uniquely identify any isomorphism class of triangulations using an efficiently-computable string called an \emph{isomorphism signature} \cite{burton2011pachner}. Every triangulation has a unique isomorphism signature, and two triangulations have the same signature if and only if they are isomorphic.

An important tool in the study of triangulations is the \emph{dual graph} (also known as the \emph{face pairing graph}). Given a triangulation $\tri$, its dual graph $\Gamma(\tri)) = (V,E)$ is the multigraph whose nodes are the pentachora in $\tri$, and for each gluing that identifies two tetrahedral facets of $p_i$ and $p_j$ we add an arc between the corresponding nodes in $V$. By construction, $\Gamma(\tri)$ has maximum degree $\leq 5$, and is $5$-regular when $\tri$ triangulates a closed $4$-manifold~$\Gamma(\tri)$. Figure \ref{fig:dualGraphs} depicts a typical dual graph.

The dual graph as defined above does not retain any information about the permutations used in the face identifications, and so $\tri$ cannot be reconstructed from $\Gamma(\tri)$ alone. Nevertheless, some information about the underlying topology of the $4$-manifold can still be extracted from $\Gamma(\tri)$. It also proves to be useful as a visual tool when discussing triangulations. 

\begin{figure}[h]
	\centering
	\begin{subfigure}{\textwidth}
	\centering
	\includegraphics{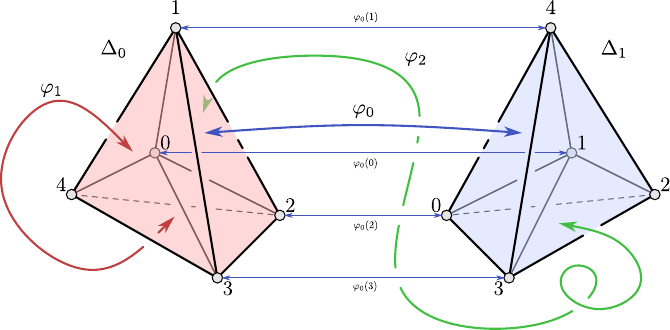}
	\caption{$\mathcal{T}$}
	\label{fig:triangulations}
	\end{subfigure}\\
	\hfill
	\begin{subfigure}{\textwidth}
		\centering
		\includegraphics{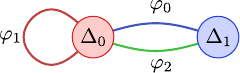}
		\caption{$\Gamma(\mathcal{T})$}
		\label{fig:dualGraphs}
	\end{subfigure}%
	\caption{(a) A triangulation $\mathcal{T}=\widetilde{\Delta}/\Phi$ with two pentachora $\widetilde{\Delta}=\{\Delta_0,\Delta_1\}$ and three face gluings $\Phi=\{\varphi_0,\varphi_1,\varphi_2\}$. The map $\varphi_0$ is given by $\Delta_0(0123)\xrightarrow{\varphi_0}\Delta_1(1403)$. (b) The dual graph $\Gamma(\mathcal{T})$ of the triangulation $\mathcal{T}$.}
\end{figure}

\begin{housekeeping}
In order to avoid confusion, we use the terms \emph{vertices} and \emph{edges} exclusively when referring to triangulations, and \emph{nodes} and \emph{arcs} when referring to graphs.         
\end{housekeeping}

\subsection{Local Moves and the Pachner Graph}\label{sec:prelims-localMoves}
We modify our triangulations using \emph{Pachner moves} (or \emph{bistellar flips}), which are local moves that change a triangulation but not its underlying PL-type \cite{pachner1987}. Informally, an $(i,j)$ Pachner move can be thought of as taking a subcomplex of $i$ pentachora in the boundary comlpex of the $5$-simplex $\partial \Delta^5$, and replacing them with its complement in $\partial \Delta^5$ of $6-i = j$ pentachora. This necessarily means that {\em (a)} the $i$ pentachora share a common $(5-i)$-dimensional face -- which is removed from the triangulation; and {\em (b)} the $j$ pentachora share a common $(i-1)$-face -- which is newly inserted into the triangulation. See \Cref{fig:pachner} for an illustration of Pachner moves in dimension four, and see \cite{pachner1987} for more details.

The use of Pachner moves has several key benefits: Firstly, two PL manifolds $X$, and $X'$ are PL-homeomorphic if and only if there exists a sequence of Pachner moves between $X$ and $X'$. This statement is known as {\em Pachner's theorem} and was first proven in \cite{pachner1987}. We can therefore certify that two triangulations are PL-homeomorphic by finding a connecting sequence of Pachner moves. Even without a guarantee that this algorithm will terminate, it is effective in practice. Secondly, Pachner moves can be used to reduce the size of a triangulation (without changing the underlying PL-type). Finally, Pachner moves are already implemented and ready-to-use in several software packages. Here, we use \emph{Regina} \cite{Regina}.
    
The \emph{Pachner graph} of a PL manifold $X$, denoted by $\mathscr{P}(X)$, is used to describe and track how distinct triangulations of a $4$-manifold can be related via Pachner moves. It is the (infinite) graph with nodes corresponding to isomorphism classes of triangulations of $X$; and two nodes of $\mathscr{P}(X)$ are joined by an arc if and only if there is a single Pachner move that takes one triangulation to the other. We refer the reader to \cite{burton2011pachner} for further details.

\begin{figure}[ht]
    \centering
    \includegraphics[width=\linewidth]{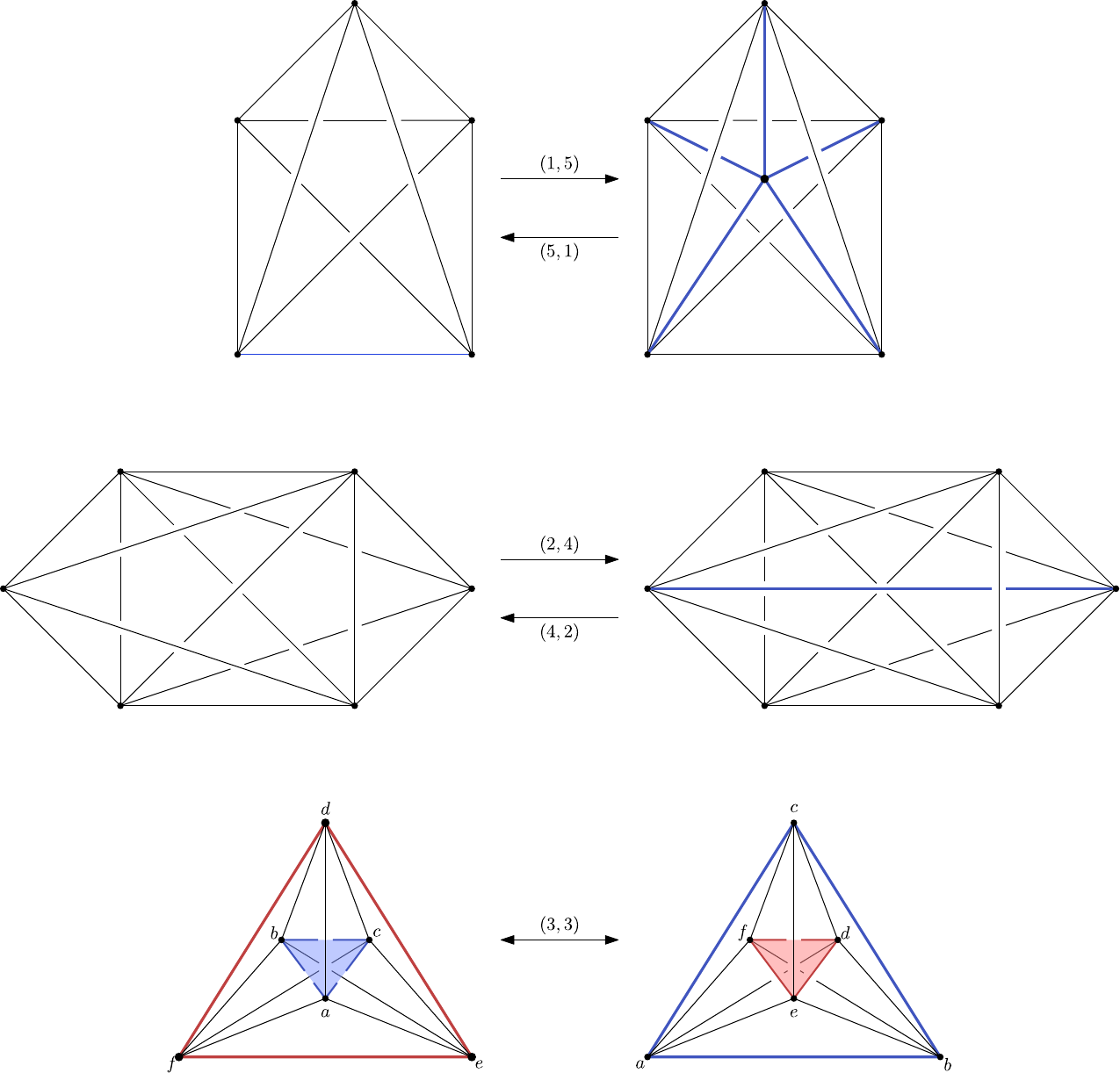}
    \caption{Pachner moves and their inverses in dimension 4}
    \label{fig:pachner}
\end{figure}

There are many other local modifications which can be expressed by sequences of Pachner moves. When moving through the Pachner graph, using these additional modifications can be powerful, since they allow the use of `short-cuts' into different areas of the graph. In this article, we make extensive use of two of these additional moves, the \emph{$2$-$0$-Edge move} and the \emph{$2$-$0$-Triangle move}. The $2$-$0$-Edge move takes two pentachora, identified along three tetrahedra to form a `pillow' around a common edge, and flattens them to form two tetrahedra. The $2$-$0$-Triangle move is similar. Here, two pentachora, are identified along two tetrahedra to form a pillow around a common triangle, and flattens them to form three tetrahedra.

\subsection{4-Manifolds}\label{sec:prelims-4mflds}

Let $\tri$ be a $4$-manifold triangulation. For the {\em ring of coefficients} $\bz$, the {\em group of $p$-chains}, $0 \leq p \leq 4$, denoted $\Chains_p(\tri,\bz)$, of $\tri$ is the group of formal sums of $p$-dimensional faces with $\bz$ coefficients. The \emph{boundary operator} is a linear operator $\partial_p: \Chains_p(\tri,\bz) \rightarrow \Chains_{p-1}(\tri,\bz)$ defined by
\[
\partial_p \sigma = \partial_p \{v_0, \cdots , v_p\} = \sum_{j=0}^p \{v_0,\cdots ,\widehat{v_j}, \cdots,v_p\},
\]
where $\sigma$ is a face of $\tri$, $\{v_0, \ldots, v_p\}$ represents $\sigma$ as a face of a pentachoron of $\tri$ in local vertices $v_0, \ldots, v_p$, and $\widehat{v_j}$ means $v_j$ is deleted from the list. 

Denote by $\Cycles_p(\tri,\bz)$ and $\Boundaries_{p-1}(\tri,\bz)$ the kernel and the image of $\partial_p$ respectively. Observing $\partial_p \circ \partial_{p+1}=0$, we define the {\em $p$-th homology group} $H_p(\tri,\bz)$ of $\tri$ by the quotient $H_p(\tri,\bz) = \Cycles_p(\tri,\bz)/ \Boundaries_p(\tri,\bz)$. Each homology group $H_p$ is a finitely generated $\mathbb{Z}$-module, and is a topological invariant of the manifold triangulated by $\tri$. We write $\beta_p=\dim(H_p(\tri,\bz))$ for the dimension of (the free part of) $H_p$. Informally, $H_p(\tri,\bz)$, $0 \leq p \leq 4$, of a triangulation $\tri$ counts the number of `$p$-dimensional holes' in $\tri$. For a more thorough introduction to homology theory see \cite{hatcherAT}.

Let $X$ be a closed, oriented $4$-manifold. Representatives of classes in $H_2(X;\bz)$ generically intersect in a finite number of points (possibly after isotoping them into transverse position). The \emph{intersection form} $Q_X$ of $X$, is the symmetric, unimodular, bilinear form defined by
\[
Q_X:H_2(X;\bz)\times H_2(X;\bz)\to\bz,\quad Q_X(\alpha,\beta)=S_\alpha\cdot S_\beta:=\sum_{S_\alpha\cap S_\beta}\pm 1,
\]
where $S_\alpha$, and $S_\beta$ are $2$-chains representing the classes $\alpha,\beta\in H_2(X;\bz)$. If $X$ is smooth, $S_\alpha$, and $S_\beta$ can be chosen to be oriented surfaces embedded in $X$ \cite[Proposition~1.2.3]{GompfStipsicz}. 
    
A landmark result of $4$-manifold topology is the following classification result for simply connected (i.e.\ having trivial fundamental group\footnote{The \emph{fundamental group} of a manifold $X$, denoted $\pi_1(X)$, describes closed paths in $X$ up to homotopy, with the group operation of path concatenation.}) topological $4$-manifolds due to Freedman.
    
\begin{theorem}[Freedman \cite{Freedman-TopOf4Mflds}]\label{thm:freedman}
For every symmetric, unimodular, bilinear form $Q$, there exists a closed simply connected topological $4$-manifold $X$ such that $Q_X=Q$. If $Q$ is even, this manifold is unique up to homeomorphism. If $Q$ is odd then are exactly two different homeomorphism types of manifolds with the given intersection form, however only at most one of these admits a PL structure. 
\end{theorem}

Shortly after Freedman's result, Donaldson showed the following equally important result.

\begin{theorem}[Donaldson \cite{Donaldson}]\label{thm:donaldson}
    The symmetric unimodular bilinear form $\oplus_m[+1]$ is the only positive definite form that can be realised as the intersection form of a PL $4$-manifold.
\end{theorem}

The results by Freedman and Donaldson (together with Serre's algebraic classification of indefinite forms) imply the following key result.

\begin{theorem}\label{thm:simplyConnectedClassification}
    Two simply connected PL $4$-manifolds are homeomorphic if and only if their intersection forms have the same rank, signature, and parity.
\end{theorem}
    
\begin{housekeeping}
All manifolds are assumed to be PL/smooth, closed, connected, and orientable, unless explicitly stated otherwise.
\end{housekeeping}

\begin{example}
    The $4$-sphere $\mathbb{S}^4$ has no 2-homology and so $Q_{\mathbb{S}^4}=\varnothing$. The complex projective plane, $\cpp$, has $Q_{\cpp}=[+1]$, and the oppositely oriented manifold $\overline{\cpp}$ has $Q_{\overline{\cpp}}=[-1]$. $S^2\times S^2$ has intersection form $Q_{S^2\times S^2}=\begin{bmatrix} 0 & 1 \\ 1 & 0 \end{bmatrix}$. 
\end{example}

\subsection{Handle Decompositions and Kirby Diagrams}\label{sec:prelims-handles}
In the smooth setting, we primarily work with $4$-manifolds via \emph{handle decompositions}. Let $X,X'$ be two smooth $4$-manifolds. We say $X$ is obtained from $X'$ by attaching a ($4$-dimensional) $k$-handle, denoted $X=X'\cup_{\varphi} h^k$, if there is an embedding 
\[\varphi:S^{k-1}\times D^{4-k}\to\partial X'\]
such that $X$ is of the form 
\[X=\left[X'\sqcup D^k\times D^{4-k} \right]/\varphi(x)\sim x,\] where $D^k$ denotes the closed $k$-disk ($k\in\{1,2,3,4\}$). There always exists a Morse function $f\colon X\to\mathbb{R}$ inducing a handle decomposition $X=\bigcup_{i=0}^n X_i$ with 
\[\varnothing=X_{-1}\subset X_0\subset\cdots\subset X_n=X\] 
in which $X_i$ is obtained from $X_{i-1}$ by attaching $i$-handles \cite{milnorMorseTheory}. Since $X$ is closed and connected, it can be assumed that there is a single $0$- and $4$-handle. Hence a closed $4$-manifold $X$ is obtained from $B^4$ by attaching $1$-, $2$-, and $3$-handles and finally capping off with another $B^4$. Given such a handle decomposition, we depict $X$ by drawing the attaching regions of the handles in the boundary of the $0$-handle~($\partial B^4\cong \mathbb{S}^3\cong\mathbb{R}^3\cup\{*\}$). By a result of Laudenbach and Poenaru \cite{LaudenbachPoenaru}, $3$- and $4$-handles attach uniquely up to diffeomorphism, and so it suffices to understand how the $1$- and $2$-handles attach. We depict $1$-handles by `dotted' unknots (see Section 5.4 of \cite{GompfStipsicz} for details of this notation). A $2$-handle is attached via a map of the form $S^1\times D^2\to S^3$. These maps are determined up to isotopy by (i) an embedding $S^1\times\{0\}\to S^3$ (i.e.\ a \emph{knot}) and (ii) a choice of normal vector field on the knot. Classes of such vector fields are in non-canonical bijection with the integers \cite{GompfStipsicz}. Once a choice for $0\in\mathbb{Z}$ has been made---the so-called \emph{$0$-framing} $f_0$---any other framing differs from $f_0$ by some integral number of twists. By convention, we make the choice that the normal vector field induced from the collar of any Seifert surface $S$ of $K$ is the zero framing. Fixing an orientation of $K$ gives a well-defined notion of linking number, and if we consider a parallel push-off $K'$ of $K$ along the surface $S$ then $\lk(K,K')=0$. This is referred to as the \emph{canonical framing} (or \emph{Seifert framing}).

As such, we draw a $2$-handle as a knot decorated with an integer. A decorated link diagram of this form---dotted unknots and integer decorated links, together with a specification of how many $3$- and $4$-handles there are---is called a \emph{Kirby diagram} and gives a combinatorial encoding of a closed $4$-manifold up to diffeomorphism. Figure \ref{fig:rationalBallRn} is an example of a typical Kirby diagram.

\section{Topological Classification}
\label{sec:top-class}

In this section we describe a classification of the triangulations in the census up to topological homeomorphism. This serves as a spring board to then carry out the PL classification. We start by grouping the triangulations of the census by their homology groups. Since all manifolds under consideration are closed and orientable, we omit $H_0(X;\bz)\cong H_4(X;\bz) \cong \bz$ from the homology vector, i.e.\ we simply refer to $(H_1(X;\bz),H_2(X;\bz),H_3(X;\bz))$. In a second step, and in the case of simply connected triangulations, we use  \emph{Regina}'s built-in intersection form routine and \Cref{thm:simplyConnectedClassification} to split these groups further.

For the remaining triangulations, we simply guess their topological types, build triangulations of these manifolds (using \emph{Katie}, see \Cref{sec:katie+plRes}, and built-in functionality of \emph{Regina}), and compare them to the census manifolds using Pachner moves (that is, we establish a PL homeomorphism). Altogether, this leads to the following classification, summarised in Tables \ref{tab:2p-TOPclassTable}, \ref{tab:4p-TOPclassTable}, and \ref{tab:6p-TOPclassTable} (note that the content of the fourth column refers to work done in \Cref{sec:pl}). In Table \ref{tab:4p-TOPclassTable}, we begin to see the appearance of $S^1$-bundles over \emph{lens spaces}, denoted $L(p,q)$, which are important class of $3$-manifold obtained by gluing two solid tori together (see \cite{saveliev2012lectures} for more).

\begin{table}[ht]
        \caption{Topological classification of the closed orientable 2-pentachoron census.}
    \label{tab:2p-TOPclassTable}
    \begin{tabular}{clrr}
    	\toprule
        \# Pentachora & $4$-Manifold & \# Triangulations & \# PL Classes\\
	\midrule
        2   & $\mathbb{S}^4$ & 6 & 1\\
            & $S^3\times S^1$ & 2 & 1\\
            \bottomrule
    \end{tabular}
\end{table}

\begin{table}[ht]
        \caption{Topological classification of the closed orientable 4-pentachoron census.}
    \label{tab:4p-TOPclassTable}
    \begin{tabular}{clrr}
    	\toprule
        \# Pentachora & $4$-Manifold & \# Triangulations & \# PL Classes \\
\midrule
        4   & $\mathbb{S}^4$ & 647 & $\leq 2$\\
            & $S^3\times S^1$ & 126 & 1\\
            & $\mathbb{C}P^2$ & 4 & 1\\
            & $S^3\times S^1\#\mathbb{C}P^2$ & 3 & 1\\
            & $\mathbb{R}P^3\times S^1$ & 1 & 1\\
            & $L(3,1)\times S^1$ & 1 & 1\\
            & $L(3,1)\twprod S^1$ & 2 & 1\\
            \bottomrule
    \end{tabular}
\end{table}

\paragraph*{Rational Homology 4-Spheres with Finite Fundamental Group}

Before giving the table for six pentachora, we first define a particular family of $4$-manifolds. Let $R(n)$ denote the $4$-manifold given by the Kirby diagram in Figure \ref{fig:rationalBallRn}, in which there are $n$ strands wrapping around the $1$-handle. 

\begin{figure}[ht]
    \centering
    \includegraphics[width=0.35\linewidth]{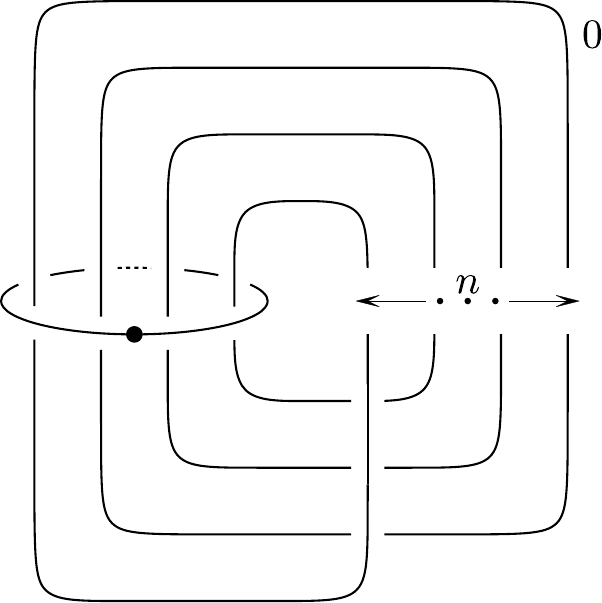}
    \caption{A Kirby diagram of the rational ball $R(n)$}
    \label{fig:rationalBallRn}
\end{figure}

This $4$-manifold $R(n)$ is a \emph{rational homology ball}---that is, it has the same homology as a $4$-sphere when one computes the homology with rational coefficients. Let $QS^4(n)$ denote the double of $R(n)$, that is 
\[QS^4(n):=DR(n)=R(n)\cup_{\mathrm{id}}R(n).\] 
The $4$-manifold $QS^4(n)$ is then a \emph{rational homology $4$-sphere} with $\pi_1(QS^4(n))\cong\bz_n$ and homology vector $(\bz_n,\bz_n,0)$. By using a combination of \emph{Katie}, \textsc{Up-Side-Down-Simplify} (see \Cref{sec:pl}), and \emph{Regina}, we obtain representative triangulations of $QS^4(n)$ for $n\in\{2,3\}$, each with six pentachora respectively.

\begin{remark}
    We note that the topological type of $QS^4(n)$ was not identified in \cite{perez}.
\end{remark}

\begin{table}[ht]
    \caption{Topological classification of the closed orientable 6-pentachoron census.}
\label{tab:6p-TOPclassTable}
    \begin{tabular}{clrr}
    	\toprule
        \# Pentachora & $4$-Manifold & \# Triangulations & \# PL Classes\\
\midrule
        6   & $\mathbb{S}^4$ & 405\,188                  &   $\leq 4 $   \\
            & $S^3\times S^1$ & 29\,124                  &   1   \\
            & $\mathbb{C}P^2$ & 4\,423                   &   $\leq 3$  \\
            & $S^2\times S^2$ & 5                       &   1   \\
            & $S^2\twprod S^2$ & 7                      &   1   \\
            & $\#_2\mathbb{C}P^2$ & 8                   &   1   \\
            & $S^3\times S^1\#\mathbb{C}P^2$ & 1\,477    &   1   \\
            & $S^3\times S^1\#_2\mathbb{C}P^2$ & 6      &   1   \\
            & $\mathbb{R}P^3\times S^1$ & 42            &   1   \\
            & $L(3,1)\times S^1$ & 55                   &   1   \\
            & $L(3,1)\twprod S^1$ & 64                  &   1   \\
            & $L(4,1)\times S^1$ & 3                    &   1   \\
            & $L(4,1)\twprod S^1$ & 3                   &   1   \\
            & $L(5,2)\times S^1$ & 1                    &   1   \\
            & $L(5,2)\twprod S^1$ & 1                   &   1   \\
            & $QS^4(2)$ & 84                            &   $\leq 2$\\
            & $QS^4(3)$ & 4                             &   1\\
\bottomrule
    \end{tabular}
\end{table} 

\section{The PL classification algorithm}
\label{sec:pl}

In this section we present our search heuristic to establish piecewise linear homeomorphisms between large quantities of triangulated manifolds conjectured to be in the same PL class. We first start by going over some subroutines before sketching the algorithm as a whole.

\subsection{Important subroutines}

\textsc{Up-Side-Down-Simplify} (USDS): This subroutine is the heart of our algorithm. It is relatively easy to describe, but details are very important. We point out that it is difficult to design a method which is efficient for both triangulations that can trivially be merged with a different class, as well as pathological cases needing millions, if not billions of moves to escape a local area of the Pachner graph (apart from the triangulation presented in \Cref{sec:pathological} resisting classification altogether, there are four more $4$-pentachoron triangulations requiring a large number of moves before being classified as standard). Extensive research has been done to find effective strategies to compare PL homeomorphism types for triangulated $4$-manifolds, see \Cref{sec:intro} for a detailed discussion. At least in the case of the census of small $4$-manifold triangulations, the below strategy yields the best results on a consistent basis.

Our heuristic is a slight variation of a standard biased Markov chain-style random walk through the Pachner graph of a triangulation. We bias the sizes of the triangulations visited by the random walk around a fixed target size of $\hat{n}$ pentachora (in our calculations, we used $8 \leq \hat{n} \leq 12$). If the current state of the random walk has more than $\hat{n}$ pentachora, an exponential penalty is imposed on choosing a Pachner move further increasing the size of the state. The same is true for a current state of smaller size than $\hat{n}$ and moves further reducing its size.

The core idea of our approach lies in the type of moves we choose to increase and decrease the sizes of our triangulations: we use $2$-$4$-Pachner moves to increase the size of a triangulation, but $2$-$0$- edge- and triangle-moves to reduce the size of a triangulation. Moreover, we use standard $3$-$3$-Pachner moves to change triangulations while keeping their $f$-vector constant. 

Empirical evidence suggest, that this approach mixes triangulations much faster than by just using standard Pachner moves. This approach has already been used very successfully by the first author \cite{Burke-SoftwareJoCGver} in the search of minimal triangulations. The method used there is even simpler in structure than the one presented below.

Our base method has four input parameters that stay fixed for the duration the method is run: {\em (a)} A set probability $x \in [0,1]$ to decide whether a $3$-$3$-move or some other move is performed,  {\em (b)} a target size $\hat{n}$ for triangulations to be visited in the random walk,  {\em (c)} a parameter $\alpha \in \mathbb{R}$ determining the severity of the penalty for sampling a triangulation of size away from $\hat{n}$, and  {\em (d)} a number of steps $s \in \mathbb{Z}$.

\begin{tcolorbox}[colframe=blue!75!green!75,title=Heuristic 1]
\textbf{Setup:}
  \begin{itemize}
      \item $x \in [0,1]$, $\alpha \in \mathbb{R}$, $\hat{n} \in \mathbb{Z}$, $s \in \mathbb{Z}$
  \end{itemize}
\textbf{Main Loop:}
\begin{itemize}
  \item Set $\tri = \tri_0$ 
  \item For each step $1 \leq i \leq s$ and while $|\tri| \neq n$:
  \begin{enumerate}
      \item Update $\tri = \tri'$
      \item Update $\beta = \frac{e^{\alpha (\hat{n}-|\tri|)}}{1+e^{\alpha (\hat{n}-|\tri|)}} \in [0,1]$
      \item Sample $u \in \mathcal{U}([0,1])$
      \item Case $u > x$: if available, do $3$-$3$-move on $\tri$ to return triangulation $\tri'$
      \item Sample $v\in \mathcal{U}([0,1])$
      \item If $v > \beta$: if available, perform $2$-$0$-Edge- or Triangle-move on $\tri$ to return triangulation~$\tri'$
      \item Perform $2$-$4$-move on $\tri$ to return triangulation $\tri'$
  \end{enumerate}
  \item If $\tri$ and $\tri_0$ are in different classes, merge classes
\end{itemize}
\end{tcolorbox}

\textsc{Adjust-Vertex-Number}: This subroutine takes a triangulated $4$-manifold $\tri_0$ as input, and outputs a PL homeomorphic triangulation with $v$ vertices.

If $\tri_0$ has fewer than $v$ vertices, perform $1$-$5$-moves $v-f_0(\tri_0)$ times to obtain a $v$-vertex triangulation. If, on the other hand, the input triangulation has more than $v$ vertices, we randomly perform edge collapses. If not enough edge collapses are available, we run USDS in between until they become available. It is worthwhile noting that the latter case has no guarantee to terminate in general, but does very quickly in practice.

\subsection{The Algorithm}

Our implementation is based on establishing PL-homeomorphisms using Pachner's theorem \cite{pachner1987}, i.e.\ by describing sequences of Pachner moves transforming one triangulation into another. Our computations are organised in a \textsc{Union-Find} structure: at any step of our calculations, every triangulation is associated to a class of triangulations for which pairwise PL homeomorphisms have already been established. Each class has a unique representative triangulation. If a sequence of Pachner moves is found turning a triangulation from one class into a triangulation from another class, both classes are merged with the representative of the latter class, and we continue. This way, if the task is to classify $N$ triangulations, we only need to establish $O(N)$, rather than $O(N^2)$ PL homeomorphisms.

We start with a large number of triangulations $\tri_1, \ldots , \tri_N$. Here, we assume that all triangulations have the same number of $n$ pentachora, but this is not necessary for the algorithm to work. We furthermore assume that we have already established that all triangulations $\tri_i$, $1 \leq i \leq N$, are homeomorphic. Again, this is not a necessary requirement for the algorithm to work, but we require, however, that all triangulations have already been tested to have the same Euler characteristic $m$.

For $n$-pentachora triangulations of Euler characteristic $m$, it is a consequence of \Cref{eq:ds1,eq:ds2,eq:ds3} that their $f$-vector is determined by their number of vertices. We have the following algorithm.

\begin{tcolorbox}[colframe=blue!75!green!75, title=Main Algorithm]
    \begin{enumerate}
    \item Initialise \textsc{Union-Find} structure with $N$ classes $\{\tri_i \}$, $ 1\leq i \leq N$, of size $1$.
    \item Fix $v = \min \{ f_0 (\tri_i) | 1 \leq i \leq N \}$, the smallest number of vertices found in our list.
    \item For all $1 \leq i \leq N$:
    \begin{enumerate}
        \item Use \textsc{Adjust-Vertex-Number} to produce a $v$-vertex triangulation $\tri'$ 
        \item Use USDS to connect $\tri'$ to $\tri_j$, $f_0(\tri_j)=v$, $f_4(\tri_j) = n$, from the list
    \end{enumerate}
    Once this step is complete, every PL class has a $v$-vertex representative.
    \item For all classes, run USDS to connect its representative to a new class. Use a moderate time-out threshold to skip hard cases.
    \item As long as there are still more than $K$ classes, goto {\bf 6.}
    \item For remaining classes, run USDS with relaxed parameters.
    \end{enumerate}
\end{tcolorbox}

\subsection{Implementation, Timings, and Remarks}
\label{ssec:implementation}

Our implementation of the algorithm described in this section is available at\newline\url{https://github.com/raburke/Dim4Census/}.

Our guiding principle for the implementation is to minimise human intervention. We point out that we can modify arbitrarily many triangulations in parallel until they can be merged with an existing class of the \textsc{Union-Find} structure. This parallelises the bottleneck of the computation for large censuses of largely easy-to-handle triangulations. When dealing with pathologically difficult combinatorial structures, running different random walks on the same triangulation may lead to similar speed-ups. However, for the $6$-pentachoron census with its $\approx 400,000$ triangulations, parallelisation is not crucial and we hence defer its implementation to future work.

For some indications of running times, we ran our algorithm on a laptop with an 11th Gen Intel i7 processor and 32GB of RAM, with input the $405\,188$ $6$-pentachoron triangulations homeomorphic to $S^4$. We obtained the following timings, summarised in Table \ref{tab:sampleRunningTimesMain}, for first running \textsc{Adjust-Vertex-Number} (referred to as {\em Step 1}), then running USDS until only 10 classes are left over ({\em Step 2}), and then the times to merge any additional class until we reduce to 6 classes, which is when we stopped our calculations.


\begin{table}[h]
	\caption{Sample running times of the Main Algorithm with input $405\,188$ $6$-pentachoron $4$-spheres.}
	\label{tab:sampleRunningTimesMain}
\begin{tabular}{rr|rrrr|r}
\toprule
Step 1 (s) & Step 2 (s) & 9 cl (s) & 8 cl (s) & 7 cl (s) & 6 cl (s) & Total time (s) \\
\midrule
$13\,686$&$10\,752$&$88$&$4\,787$&$11\,008$&$4\,987$&$45\,308$ \\
$14\,061$&$10\,306$&$373$&$1\,097$&$335$&$1\,143$&$27\,315$ \\
$14\,170$&$9\,264$&$365$&$451$&$464$&$3\,754$&$28\,468$\\
$14\,182$&$9\,359$&$1\,116$&$1\,772$&$99$&$3\,772$&$30\,300$\\
$14\,142$&$9\,791$&$405$&$811$&$1\,598$&$7\,961$&$34\,708$ \\
$13\,212$&$9\,342$&$302$&$48$&$3\,100$&$3\,155$&$29\,159$\\
\bottomrule
\end{tabular}
\end{table}

Achieving 5 connected classes usually does not take much longer (with $1\, 611$ and $18\, 929$ additional seconds in two of the six runs summarised above). Achieving four class takes around a week, as we were able to observe on multiple occasions. The algorithm never achieved 3 classes. Running the complete algorithm in parallel on $100+$ cores may produce additional results. Alternatively, exhaustive enumeration on a larger machine with more memory may lead to additional merges of classes.

PL classification of other topological types was achieved through a combination of exhaustive enumeration and our main algorithm. Relevant timings are summarised in Table \ref{tab:sampleRunningTimesPLClass}.


\begin{table}[h]
\caption{Sample running times of PL classification for selected manifolds.}
\label{tab:sampleRunningTimesPLClass}
\begin{tabular}{l|rrr|r}
\toprule
Manifold & Step 1 (s)& Step 2 (s)& 1 cl (s) & Total time (s) \\
\midrule
$S^3\times S^1$&$2480$&$119$&$51$&$2650$\\
$S^2\twprod S^2$&$115$&$0$&$4\,514$&$4\,629$\\
$\#_2\mathbb{C}P^2$&$27$&$0$&$1\,531$&$1\,558$\\
$S^3\times S^1\#\mathbb{C}P^2$&$1\,060$&$140$&$42$&$1\,242$\\
$S^3\times S^1\#_2\mathbb{C}P^2$&$0$&$0$&$62\,022$&$62\,022$\\
\bottomrule
\end{tabular}
\end{table}

\subsection{{\em Katie} and PL Classification Results}\label{sec:katie+plRes}

As discussed in \Cref{sec:pl}, our PL classification algorithm determines whether two given triangulations are PL homeomorphic. What is missing from this algorithm is a reference triangulation for which its PL homeomorphism type is known.

For this we use the software tool \emph{Katie} \cite{Katie,Burke-SoftwareJoCGver},  developed by the first author. \emph{Katie} is based on an algorithm due to Casali and Cristofori \cite{casaliCristofori2023Final}. It takes as input a Kirby diagram and produces a triangulation of the associated PL $4$-manifold. For each of the topological types in \Cref{tab:2p-TOPclassTable,tab:4p-TOPclassTable,tab:6p-TOPclassTable}, we start with the Kirby diagram representing its canonical PL-type, and compare the resulting triangulation to the triangulations in the census. 

For all but three topological types, all triangulations homeomorphic to a given $4$-manifold, are pairwise PL-homeomorphic. The three exceptions are $\mathbb{S}^4$, $\cpp$, and the $\mathbb{Q}$-homology sphere $QS^4(2)$, cf. \Cref{tab:6p-TOPclassTable} for which we find $4$, $3$, and $2$ classes respectively.

As already documented in \Cref{ssec:implementation}, our algorithm reliably takes all $405\,188$ triangulations homeomorphic to $\mathbb{S}^4$ from the $6$-pentachoron census and classifies them into around $5$ PL classes within a day of computation time. Some of the remaining classes have remarkable combinatorial properties, which are discussed in more detail in \Cref{sec:pathological}. These remaining classes are certified to be impossible to connect to other classes using an exhaustive enumeration and traversal of the Pachner graph with an excess height of $6$ (that is, only triangulations with number of pentachora up to $6+6=12$ are considered). 

In some contrast to the triangulations homeomorphic to $\mathbb{S}^4$, we can connect all but four of the $29\, 124$ triangulations homeomorphic to $S^3\times S^1$ using exhaustive simplification and traversal of the Pachner graph, with an excess height of at most four. The remaining four can be confirmed to be standard $S^3\times S^1$s by (i) retriangulating to show they were all PL-homeomorphic to each other, and then (ii) traversing the Pachner graph with an excess height of at most $6$. Alternatively, using our algorithm, all $29\, 124$ triangulations homeomorphic to $S^3\times S^1$ can be connected within an hour of computation time, see \Cref{ssec:implementation}.

This difference in behaviour is noteworthy since $\mathbb{S}^4$ and $S^3\times S^1$ are the only two closed orientable $4$-manifolds which can be triangulated using only $2$ pentachora.

\section{Pathological Triangulations, Combinatorial Obstructions, and 2-Knots}
\label{sec:pathological}

In this section we will present some preliminary analysis of, and discuss ideas concerning, the triangulations still resisting classification. We focus our attention for now on the unique $4$-pentachoron $4$-sphere which we were unable to connect to the standard $4$-sphere. The sphere in question, which we will denote by $Q$, has isomorphism signature \is{eAMPcaabcddd+aoa+aAa8aQara}. The dual graph of $Q$ is depicted in Figure~\ref{fig:Qara-DualGraph}. The colouring of the graph will be explained as required throughout the section.

\begin{figure}[h]
	\centering
	\includegraphics[width=0.5\textwidth]{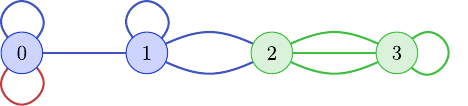}
		\caption{The dual graph of $Q$}
	\label{fig:Qara-DualGraph}
\end{figure}

The $2$-pentachoron subcomplex consisting of pentachora $2$ and $3$ of $Q$ also appears within many of the other triangulations resisting classification (for example, the $6$-pentachoron $\mathbb{C}P^2$s). The subcomplex in question, denoted $C$, has isomorphism signature \is{cHIbbb0bRbpb}, $f$-vector $f=(1,1,5,6,2)$, and boundary an ideal $2$-tetrahedron solid torus. 

Considered in isolation $C$ is not a triangulation of a manifold: its single vertex link is a pinched genus-$4$ handlebody, and the link of its single edge is a thrice-punctured sphere. However, the boundary of $C$ can be uniquely closed up to produce an ideal triangulation of a $4$-manifold. Performing the gluing $\Delta_0(0234)\leftrightarrow\Delta_0(0312)$ yields  the unique\footnote{Verifiable via a very short exhaustive search of the ideal $2$-pentachoron census.} $2$-pentachoron ideal triangulation of a \emph{Cappell--Shaneson $2$-knot} complement, first analysed by Budney--Burton--Hillman \cite{budneyBurtonHillman-CappellShanesonComp}, and which we will denote by $CS$. This is an ideal triangulation of the complement of a \emph{knotted $2$-sphere} $\Sigma$ in the $4$-sphere, $S^4-\nu \Sigma$ (where $\nu\Sigma$ denotes a tubular neighbourhood of $\Sigma$). By knotted, we mean that $\Sigma$ does not bound an embedded $3$-ball in $S^4$. The Cappell--Shaneson $2$-knot triangulation considered in \cite{budneyBurtonHillman-CappellShanesonComp} is just one in a family of knotted $2$-spheres in $S^4$, which are in a sense parametrised by elements of $GL(\mathbb{Z}^3)$ and their traces \cite{CappellShaneson}. The $2$-knot of \cite{budneyBurtonHillman-CappellShanesonComp} corresponds to one such $A\in GL(\mathbb{Z}^3)$ with trace $0$.
 
In addition to the study of $2$-knots being an interesting pursuit in its own right, $2$-knots have also appeared as an important source of examples for constructing potential counterexamples to S4PC via the \emph{Gluck construction}, which goes as follows. Given a knotted $2$-sphere $\Sigma$ in $S^4$, remove a tubular neighbourhood of $\Sigma$ (i.e.\ a copy of $\Sigma\times D^2$) and reglue it by the self-diffeomorphism of $S^2\times S^1$ which rotates the $S^2$ factor once as one travels around $S^1$. The result is a topological $4$-sphere but which in general is not known to be diffeomorphic to the standard $4$-sphere. However, it has been shown that for many classes of $2$-knots, the result is indeed a standard $S^4$: for example, twist-spun knots (discussed further in \Cref{sec:future-2knots}) \cite{Gordon-KnotsInS4,Pao-TwistingSpunKnots},
doubles of ribbon disks (Exercise 6.2.11(b)~\cite{GompfStipsicz}), and various Cappell--Shaneson $2$-knots \cite{akbulut-csSpheresAreStandard,Gompf-KillingAkbulutKirbySphere,Gompf-csSpheresAreStandard,Iwaki-CSspheresAreStandard}.

The `cut-open' Cappell--Shaneson complex $C$ embeds in a topological $4$-ball $B^4_C$, which is itself obtained by ungluing a single facet identification of $Q$ (depicted by the red arc in Figure \ref{fig:Qara-DualGraph}). The isomorphism signature of $B^4_C$ is~\texttt{eGzMkabcdddcaGa8aAa0awa}.
Perhaps the first point of interest with regards to $B^4_C$ is that even the presentation of its fundamental group is challenging to simplify, with a presentation given by 
\[
    \langle a, b \,\mid\, a^3 b^3 a^{-2} b^{-2}, a^{-3} b^{-1} a^5 b^2 \rangle.
\]
Using GAP, we are able to verify that the order of the group is 1, and hence is indeed the trivial group. We note that GAP uses coset enumeration to compute the order in this case, rather than performing any kind of simplification heuristic, for example via a sequence of Tietze transformations. 

Moreover, it is worthwhile to note that when $B^4_C$ is attached to a larger triangulation, the $B^4_C$ subcomplex remains intact through very large flip sequences. For example, attaching $B^4_C$ to a minimal triangulation of $S^2\times S^2-\mathring{D^4}$ produces a new triangulation of a topological $S^2\times S^2$ which we are unable to (re)simplify to the original triangulation of $S^2\times S^2$. Similarly for many other choices of $4$-manifold.

In this sense, $C$, or $B^4_C$, is a potential combinatorial obstruction to simplifying triangulations it appears within (similar to the topological `dunce hat' being a combinatorial obstruction to collapsing a contractible space). The fact that all five triangles of $C$ are in fact dunce hats perhaps gives some, albeit naive, evidence in support of this idea. In other words, one might wonder if $C$ constitutes a kind of higher-dimensional analogue of the dunce hat.

Whilst the $4$-sphere $Q$ does not appear to be directly obtained from the Cappell--Shaneson $2$-knot itself via the Gluck construction---since truncating the ideal vertex of $CS$ and attaching a $S^2\times D^2$ to the resulting real boundary yields triangulations of the $4$-sphere which we are able to successfully simplify to a minimal triangulation of the standard $4$-sphere---it nevertheless still seems tempting to assume that $Q$ relates in some way to the Cappell--Shaneson $4$-spheres, in light of the complex $C$ and its relation to $CS$.

We have the following results and conjectures concerning $B^4_C$ and $Q$.

\begin{conjecture}
The topological $4$-ball $B^4_C$ is PL-homeomorphic to the standard $4$-ball.
\end{conjecture}

We conjecture that $B^4_C$ supports the standard PL structure because of its small size and previously discussed connection to Cappell--Shaneson sphere which is known to be standard.

\begin{proposition}
    If there exists a sequence of Pachner moves connecting $Q$ to a triangulation of the standard PL $4$-sphere, then this sequence must contain at least one triangulation with at least $12$ pentachora.
\end{proposition}

\begin{proof}
  Running an exhaustive enumeration of all local move sequences starting from $Q$ and visiting triangulations up to $10$ pentachora (i.e.\ an excess height of $6$) does not connect $Q$ to any other $4$-pentachoron triangulation of the $4$-sphere.
\end{proof}

\section{Ongoing and Future Work}
To conclude this article, we detail how to extend the observations made in the previous section and, more broadly, in this article, to two research directions for future work.
    
\subsection{The 8-Pentachoron Census}\label{sec:future-eightP}
Generating a census (of any type) with $n$ simplicies consists of two stages: First, enumerate all possible dual graphs with $n$ nodes; then, for each graph, test all possible $5n/2$ facet gluings and retain those that yield valid triangulations.
For 2, 4, and 6 pentachora there are $3$, $26$, and $638$ such graphs, 
leading to $8$, $784$, and $440\,495$ valid 4-manifold triangulations, respectively.

We can deduce from this that only a very small fraction of possible gluings give rise to a triangulation of a 4-manifold, implying that various optimisations are possible (and needed) to build the census even only up to $6$ pentachora (cf.\ \cite{burton-efficientEnumeration,burtonNonOrientableCensusUnionFind,burtonPettersson-enumerating3} for an overview of such optimisations in the $3$-dimensional setting). In the case of 8 pentachora however, existing optimisations appear to no longer be sufficient. Consequently, entirely new algorithms are in development to complete this next step of the $4$-dimensional census. A completed $8$-pentachoron census will come with new challenges for classifying their PL-types. Undoubtedly, more interesting pathological triangulations will emerge and in greater numbers. This is work in progress.

\subsection{More 2-Knot Complements -- Generation and Classification}\label{sec:future-2knots}
The Cappell--Shaneson $2$-knot complement discussed at the beginning of this section is only one of many such topological objects that can be described using only a small number of pentachora. In theory, each of them can lead to difficult triangulations of $4$-balls and $4$-spheres, related to $B^4_C$ and $S^4_C$. Such examples are very useful for constructions in low-dimensional topology, or to benchmark future iterations of search methods. Given such a collection of examples, the major challenge is to either rigorously quantify how difficult they are to simplify, or to relate them to former or present potential counterexamples to S4PC. 

It would therefore seem worthwhile to attempt to enumerate (and ideally, also classify) triangulations of $2$-knot complements. This would provide us with a wealth of examples with which to construct interesting $4$-sphere triangulations.  

Whilst we currently do not have a means of constructing triangulations of arbitrary $2$-knot complements, we do have several sources of examples and construction techniques for certain classes of $2$-knots. One of the first techniques for constructing non-trivial $2$-knots was the \emph{spinning construction} due to Artin \cite{ArtinSpin}, which can be described as follows. Let $K:S^1\to\mathbb{R}^3$ be a classical $1$-knot. One can always isotope $K$ such that it lies in the upper half space $\mathbb{R}^3_+=\{(x,y,z):z\geq 0\}$ except for an unknotted arc which lies below the $xy$-plane. Remove the interior of this unknotted arc to obtain a knotted arc in $\mathbb{R}^3_+$ with its endpoints in the $x$-$y$ plane. Now rotate $\mathbb{R}^3_+$ about $\mathbb{R}^2$ through $\mathbb{R}^4$ via the map which sends a point $(x,y,z)\in\mathbb{R}^3_+$ to $(x,y,z\cos\theta,z\sin\theta)\in\mathbb{R}^4$ ($0\leq\theta< 2\pi$). In this way, the knotted arc sweeps out a knotted $\mathbb{S}^2$ in $\mathbb{R}^4$ (moreover the $2$-knot obtained through this method is independent of the arc removed from the original $1$-knot).

The latest release of \emph{Regina} (version 7.4) now includes functionality to produce a triangulation of the complement of such a {\em spun knot}, given a $1$-knot as input. This provides us with as many triangulations of $2$-knot complements as there are $1$-knots.  

A generalisation of the spinning construction, due to Zeeman \cite{Zeeman-TwistingSpunKnots}, involves also rotating the knotted arc itself independently by a whole number of twists whilst the `sweep-out' through $\mathbb{R}^4$ takes place; this gives the so-called \emph{$k$-twist spun} of the knot $K$, denoted $\tau_k(K)$. As of the time of writing, we have only implemented an algorithm to construct `regular' (i.e.\ $0$-twist) spun knots in \emph{Regina}, and so it is a point of future work to extend this construction to the $k$-twist case.

The second source of examples is the ideal census with up to $6$-pentachora which has also already been generated (though not yet classified in its entirety). It is not difficult to identify potential candidates for $2$-knot complements within the census: a triangulation of a $2$-knot complement has the same homology as $\mathbb{S}^1$, and has boundary $S^2\times S^1$. If we wish to further filter out potential `unknots' (i.e.\ triangulations of $S^1\times B^3$) then we filter out triangulations with fundamental group isomorphic to $\mathbb{Z}$ and look for more interesting fundamental groups. From such a list of candidates, in order to conclude that a given triangulation is indeed the complement of an embedded $S^2$ in a $4$-sphere, one needs to verify that attaching a $S^2\times D^2$ to the boundary of a given candidate produces a $4$-sphere. 

Using this process, we were able to (i) completely classify the ideal census with $2$ pentachora (Table \ref{tab:ideal2pClassification}); (ii) find complements of non-trivial $2$-knots with $4$ pentachora; and (iii) find candidates for complements of non-trivial $2$-knots with $6$ pentachora. Note that in the $2$-pentachoron census presented in Table \ref{tab:ideal2pClassification}, since the triangulations are of manifolds with ideal boundary, the manifolds labelled $L(p,q)\times I$ can also be understood as double cones over $L(p,q)$. 
All triangulations within a given row were determined to be PL-homeomorphic and hence we do not list the number of PL classes.

\begin{table}[h]
	\caption{Classification of the ideal orientable $2$-pentachoron census.}
	\label{tab:ideal2pClassification}
	\begin{tabular}{lr}
		\toprule
		$4$-Manifold & \# Triangulations\\
		\midrule
 $S^1\times B^3$ & 3\\
 Cappell--Shaneson Trace $0$ & 1 \\
 $L(3,1)\times I$ &  1\\
 $L(4,1)\times I$ &  1\\
 $L(5,2)\times I$ & 1 \\
	\bottomrule
	\end{tabular}
\end{table}

Of the $3\,366$ triangulations within the ideal $4$-pentachoron census, we identified $1\,754$ triangulations of $S^1\times B^3$ and $54$ non-trivial $2$-knot complements. 
Note that in general, unlike $1$-knots, $2$-knots are not determined by their complement (nor fundamental group) \cite{CappellShaneson}; hence names in Table \ref{tab:4p2Knots} should be treated as an `alias' for convenience (and indeed for the final two rows of Table \ref{tab:4p2Knots} this reference label is largely meaningless as we have yet to determine a known label for these $2$-knots; the subscripts refer vaguely to the fundamental group). 

Issa \cite{Issa} has produced triangulations of Cappell--Shaneson complements up to trace $20$ (along with the associated $4$-spheres), which enabled us, by direct comparison, to identify the particular Cappell--Shaneson knots within the census. 
Similarly, again by direct construction and comparison, we were able to identify the $2$-twist spins of the trefoil and figure-eight knots, denoted by $\tau_2(3_1)$ and $\tau_2(4_1)$ respectively, in the table. Since \emph{Regina} does not currently support $k$-twist spun knots for $k>0$, the method of construction we employed was by exploiting the fact that the $2$-twist spin of a $2$-bridge knot $K(p,q)$ is fibred by the punctured lens space, which we denote by $L(p,q)^\circ$ \cite{Zeeman-TwistingSpunKnots,Teragaito-fibered2KnotsLens}. 

By first constructing $L(p,q)^\circ$ and then using \emph{Regina}'s \texttt{bundleWithMonodromy} function we obtained triangulations with the correct topology. After simplifying, we were then able to demonstrate PL-homeomorphisms with those triangulations which appeared in the census. 
It is known that fibred $2$-knots with fibre a punctured lens space are in fact determined by their complement~\cite{PlotnickSuciu-fiberPuncturedLensDetExt}. 

We make the following conjectures regarding the identities of the complements labelled by $K_{3*3}$ and $K_{3,3}$. 
Firstly, we conjecture that $K_{3*3}$ is the complement of a \emph{$2$-cable of the $2$-twist spun trefoil} (see \cite{Teragaito-fibered2KnotsLens} for the definition of cabling in the context of $2$-knots). 
Second, $K_{3,3}$ is the complement of a \emph{$1$-roll spun Figure-8} \cite{fox-rolling}. 
Evidence for both of these claims come from an analysis of their fundamental groups.
In the case of $K_{3*3}$, its fundamental group is
\[
	\pi_1\cong\mathbb{Z}\rtimes_{\varphi}(\mathbb{Z}_3*\mathbb{Z}_3),
\]
where $\varphi(1)=\psi\in\mathrm{Aut}(\mathbb{Z}_3*\mathbb{Z}_3)$ with $\psi(a)=b$ and $\psi(b)= a^{-1}$. 
This aligns with what we would expect to see from a $2$-cable of the $2$-twist spun of the trefoil---a fibered $2$-knot with fibre $(L(3,1)\# L(3,1))^\circ$.

For $K_{3,3}$, one of the triangulations has fundamental group with presentation
\[
	\langle a,b\,\vert\, aba^{-1}bab^{-1}=1,\,a^3b^3=1\rangle,
\]
which can be rewritten as
\[
	\langle a,b\,\vert\,aba=bab\,a^3=b^3\rangle,
\]
which was shown to be the fundamental group of the $1$-roll spun Figure-8 \cite{fox-rolling}.
It may be possible to verify these claims constructively in the same manner as with the $2$-twist spun examples, however the limited construction algorithms currently available make this strategy difficult to realise.
Developing algorithms allowing us to triangulate arbitrary $n$-twist $m$-roll spun knots is a point for further research.

The triangulations are available to study via a \emph{Regina} data file in the GitHub repository (\url{https://github.com/raburke/Dim4Census/}).

\begin{table}[h]
	\caption{Table of $4$-pentachoron $2$-knot complements}
	\label{tab:4p2Knots}
	\begin{tabular}{lr}
		\toprule
	Label & \# Triangulations \\
	\midrule
	$S^1\times B^3$ & $1\,754$\\
	Cappell--Shaneson Trace 0 & 38\\
	Cappell--Shaneson Trace 1 & 2\\
	Cappell--Shaneson Trace 2 & 1\\
	$\tau_2(3_1)$ & 3\\
	$\tau_2(4_1)$ & 1\\
	$K_{3*3}$ & 1\\
	$K_{3,3}$ & 8\\
	\bottomrule
	\end{tabular}
\end{table}

A preliminary search also revealed that out of the $405\,188$ $6$-pentachoron triangulations of the $4$-sphere, $256$ of them contained a $4$-pentachoron subcomplex corresponding to a `cut open' $2$-knot complement (in the same vein as in the case of $C$ and $Q$). A point of future work is to determine how difficult these particular spheres are to simplify.

In terms of the $6$-pentachoron census, as of the time of writing we have only carried out a very coarse preliminary filtering and classification. The ideal $6$-pentachoron census contains some $2\,787\,568$ triangulations. Of these we identified $1\,088\,164$ triangulations homeomorphic to $S^1\times B^3$ and $8\,467$ candidates for non-trivial $2$-knot complements. Using the new \texttt{spun} function in \emph{Regina} we were able to identify 7 PL-homeomorphic triangulations of the ($0$-twist) spun trefoil, $\tau_0(3_1)$. Similarly, we found zero triangulations PL-homeomorphic to $\tau_0(4_1)$ (and our current best efforts to obtain a small triangulation of $\tau_0(4_1)$ yields a triangulation with 16 ideal pentachora). In light of these last two points and the observations from the $4$-pentachoron ideal census---namely the triangulations of $\tau_2(3_1)$ and $\tau_2(4_1)$---we conclude with the following question. 

\begin{question}
	Does the complement of $\tau_2(K(p,q))$ always require fewer pentachora to triangulate than $\tau_0(K(p,q))$?
\end{question}

We hope for the preliminary results and ideas discussed in this section to be the subject of future work dedicated specifically to triangulations of $2$-knot complements and related constructions.

{
\interlinepenalty=10000 
\bibliography{references}
}
\end{document}